\documentclass[a4paper,12pt,oneside]{amsart}

\usepackage{a4wide}
\usepackage[utf8]{inputenc} 
\usepackage[T1]{fontenc} 
\usepackage{textcomp} 
\usepackage[english,french]{babel} 
\usepackage{amsmath}
\usepackage{amssymb}
\usepackage{amsthm}
\usepackage{graphicx}
\usepackage{url}
\usepackage{hyperref}
\usepackage[all]{xy}
\usepackage{nccmath}
\usepackage{tikz-cd}
\usepackage{color}
\usepackage{mathrsfs,xspace}
\usepackage{comment}
\usepackage{extarrows}
\usepackage{enumitem}
\usepackage{slashbox}
\usepackage[font=small,labelfont=bf]{caption}
\allowdisplaybreaks

\makeatletter
\newcommand*\bigcdot{\mathpalette\bigcdot@{.8}}
\newcommand*\bigcdot@[2]{\mathbin{\vcenter{\hbox{\scalebox{#2}{$\m@th#1\bullet$}}}}}
\makeatother

\DeclareMathOperator{\Z}{\mathbb{Z}}
\DeclareMathOperator{\C}{\mathbb{C}}

\DeclareMathOperator{\N}{\mathbb{N}}
\DeclareMathOperator{\RR}{R}
\DeclareMathOperator{\Hol}{\mathcal{O}}
\DeclareMathOperator{\CC}{\mathcal{C}}
\DeclareMathOperator{\HH}{\mathcal{H}}
\DeclareMathOperator{\Ind}{Ind}
\DeclareMathOperator{\BM}{BM}

\DeclareMathOperator{\Db}{\mathcal{D}\textit{b}}
\DeclareMathOperator{\supp}{supp}
\DeclareMathOperator{\PPP}{\mathbb{P}}

\theoremstyle{definition}
\newtheorem{definition}{Definition}[section]
\newtheorem{example}[definition]{Example}

\newtheorem{remark}[definition]{Remark}

\theoremstyle{plain}
\newtheorem{proposition}[definition]{Proposition}
\newtheorem{theorem}[definition]{Theorem}
\newtheorem{lemma}[definition]{Lemma}

\numberwithin{equation}{section}

\begin{document}

\selectlanguage{english}

\title{Holomorphic cohomological convolution and Hadamard product}
\author{Christophe Dubussy \quad Jean-Pierre Schneiders}
\date{\today}
\address{B\^at. B37 \\ Analyse alg\'{e}brique \\ Quartier Polytech 1 \\ All\'{e}e de la d\'{e}couverte 12 \\ 4000 Li\`{e}ge \\ Belgique.}
\email{C.Dubussy@uliege.be and jpschneiders@uliege.be}
\thanks{The first author is supported by a FNRS grant (ASP 111496F)}

\begin{abstract}
In this article, we explain the link between Pohlen's extended Hadamard product and the holomorphic cohomological convolution on $\C^*$. For this purpose, we introduce a generalized Hadamard product, which is defined even if the holomorphic functions do not vanish at infinity, as well as a notion of strongly convolvable sets.
\end{abstract}

\maketitle

\tableofcontents

\section{The extended Hadamard product}

Classically, the Hadamard product of two formal power series $A(z) = \sum_{n=0}^{+\infty}a_n z^n$ and $B(z) = \sum_{n=0}^{+\infty}b_n z^n$ is defined by setting

\[
(A\star B)(z) = \sum_{n=0}^{+\infty}a_n b_n z^n.
\]
Using Taylor expansions, one can thus define the Hadamard product $f_1\star f_2$ of two germs $f_1$ and $f_2$ of holomorphic functions at the origin. Exploiting the Cauchy's integral representation, one obtains the formula

\[
(f_1\star f_2)(z) = \frac{1}{2i\pi}\int_{C(0,r)^+} f_1(\zeta) f_2\left(\frac{z}{\zeta}\right) \frac{d\zeta}{\zeta}
\]
for all $z$ in a neighbordhood of $0$, $C(0,r)^+$ being a small positively oriented circle centered at the origin (see e.g. \cite{Avan74} and \cite{Mull92} for some applications).

\bigskip

In his thesis \cite{Pohl09} (see also \cite{Mull12}), Timo Pohlen introduced the more general notion of Hadamard product for holomorphic functions defined on open subsets of the Riemann sphere $\PPP = \C \cup \{\infty\}$ which do not necessarily contain the origin. This new definition led to interesting applications, (e.g. \cite{Lors15} and \cite{Mull10}). In this introduction, we shall recall the construction and the results of T. Pohlen.

\begin{definition}
Let $\PPP$ be the Riemann sphere equipped with its canonical structure of complex manifold. Let $\Omega$ be an open subset of $\PPP$. One sets

\[
\HH(\Omega) = \{f \in \Hol(\Omega) : f(\infty)=0\}
\]
if $\infty \in \Omega$ and $\HH(\Omega) = \Hol(\Omega)$ otherwise.
\end{definition}

\begin{definition}\label{def:extmult}
We set $M = (\PPP \times \PPP) \backslash \{(0,\infty),(\infty,0)\}$ and extend the complex multiplication continuously as a map $\cdot : M \to \PPP.$ We then have

\[
\infty \cdot a = a \cdot \infty = \infty
\]
if $a\in \PPP$ is not equal to zero. If $A,B$ are subsets of $\PPP$ such that $A \times B \subset M$, one sets

\[
A\cdot B = \{a\cdot b : a \in A, b \in B\}.
\]
One also extends the inversion $z \mapsto z^{-1}$ continuously from $\C^*$ to $\PPP$ by setting $0^{-1}=\infty$ and $\infty^{-1} =0.$ If $S \subset \PPP,$ one sets

\[
S^{-1} = \{z : z^{-1} \in S\}.
\]
\end{definition}

For the rest of the article, we shall often drop the point and write the multiplication  as a concatenation.

\begin{definition}\label{def:stareligible}
Two open subsets $\Omega_1, \Omega_2\subset \PPP$ are called \emph{star-eligible} if

\begin{enumerate}
\item $\Omega_1$ and $\Omega_2$ are proper subsets of $\PPP,$
\item $(\PPP \backslash \Omega_1) \times (\PPP \backslash \Omega_2) \subset M,$
\item $(\PPP \backslash \Omega_1)(\PPP \backslash \Omega_2) \neq \PPP.$
\end{enumerate}

\noindent
\index{Star product}In this case, the \emph{star product} of $\Omega_1$ and $\Omega_2$, noted $\Omega_1\star \Omega_2,$ is defined by

\[
\Omega_1 \star \Omega_2 = \PPP \backslash ((\PPP \backslash \Omega_1)(\PPP \backslash \Omega_2)).
\]
\end{definition}

For the several equivalent definitions of the index/winding number of a cycle $c$ in $\C$, we refer to \cite{Roe15}. For any cycle $c$ in $\C$, one sets $\Ind(c,\infty) = 0.$

\begin{definition}
Let $\Omega$ be a non-empty open subset of $\PPP$, $K$ be a non-empty compact subset of $\Omega$ and $c$ be a cycle in $\Omega \backslash (K \cup \{0\} \cup \{\infty\}).$ If $\infty \notin K$ and

\[
\Ind(c,z) = \begin{cases}1 \,\, &\text{if} \,\,\,\, z\in K\\ 0 \,\, &\text{if} \,\,\,\, z\in \PPP\backslash \Omega \end{cases},
\]
then $c$ is called a \emph{Cauchy cycle} for $K$ in $\Omega.$ If $\infty \in \Omega$ and

\[
\Ind(c,z) = \begin{cases}0 &\text{if} \,\,\,\, z\in K\\ -1 &\text{if} \,\,\,\, z\in \PPP\backslash \Omega \end{cases},
\]
then $c$ is called a \emph{anti-Cauchy cycle} for $K$ in $\Omega.$
\end{definition}

In \cite{Pohl09}, Lemma $2.3.1$, T. Pohlen refers to ad hoc explicit constructions which ensure that Cauchy and anti-Cauchy cycles always exist for any $\Omega$ and any $K$. In the next section, we shall see that this existence can easily be obtained by using singular homology.

\bigskip

Let $\Omega_1$ and $\Omega_2$ be two star-eligible open subsets of $\PPP$. Note that, if $z \in \Omega_1 \star \Omega_2$, then $z(\PPP\backslash\Omega_2)^{-1}$ is a closed subset of $\Omega_1.$

\begin{definition}
Let $z \in (\Omega_1 \star \Omega_2)\backslash \{0,\infty\}.$ A \emph{Hadamard cycle} for $z(\PPP\backslash\Omega_2)^{-1}$ in $\Omega_1$ is a cycle $c$ in $\Omega_1 \backslash (z(\PPP\backslash\Omega_2)^{-1} \cup \{0\}\cup \{\infty\})$ which satisfies the condition given in the following table :

\bigskip
\bigskip

\begin{center}
\renewcommand{\arraystretch}{1} 
\setlength{\tabcolsep}{0.4cm} 
\begin{tabular}{|c|c|c|c|c|}
\hline
\backslashbox{\,$\Omega_2$}{ ${}$ \\  ${}$ \\ \,$\Omega_1$} & $0,\infty$ & $\infty$ & $0$ & \\
\hline
$0,\infty$ & $\text{cc}^+$ or $\text{acc}^-$ & $\text{acc}^-$ & $\text{cc}^+$ & cc \\
\hline
$\infty$ & $\text{acc}^-$ & $\text{acc}^-$ & / & /\\
\hline
$0$ & $\text{cc}^+$ & / & $\text{cc}^+$ & / \\
\hline
& acc & / & / & /\\
\hline
\end{tabular}
\end{center}

\bigskip
\bigskip
\noindent
This table should be understood in the following way : The elements in the first row and the first column tell which of these elements are in $\Omega_1$ and $\Omega_2$ respectively. The abreviation cc (resp. acc) means that $c$ is a Cauchy (resp. anti-Cauchy) cycle for $z(\PPP\backslash\Omega_2)^{-1}$ in $\Omega_1$. The abreviation $\text{cc}^+$ (resp. $\text{acc}^-$) means that $c$ is a Cauchy (resp. anti-Cauchy) cycle with the extra condition $\Ind(c,0)=1$ (resp. $\Ind(c,0)=-1$). A "/" means that this case cannot occur.
\end{definition}

One can now extend the standard Hadamard product.

\begin{definition}\label{def:pohlenhad}
Let $f_1 \in \HH(\Omega_1)$ and $f_2 \in \HH(\Omega_2).$ For each $z\in (\Omega_1 \star \Omega_2)\backslash \{0,\infty\}$ one sets

\[
(f_1\star f_2)(z) = \frac{1}{2i\pi}\int_{c_z} f_1(\zeta)f_2\left(\frac{z}{\zeta}\right) \frac{d\zeta}{\zeta},
\]
where $c_z$ is a Hadamard cycle for $z(\PPP\backslash\Omega_2)^{-1}$ in $\Omega_1.$ One can check that this integral does not depend on the chosen Hadamard cycle (see Lemma 3.4.2 in \cite{Pohl09}). The function $f_1 \star f_2$ is called the \emph{Hadamard product} of $f_1$ and $f_2$.
\end{definition}

\begin{center}
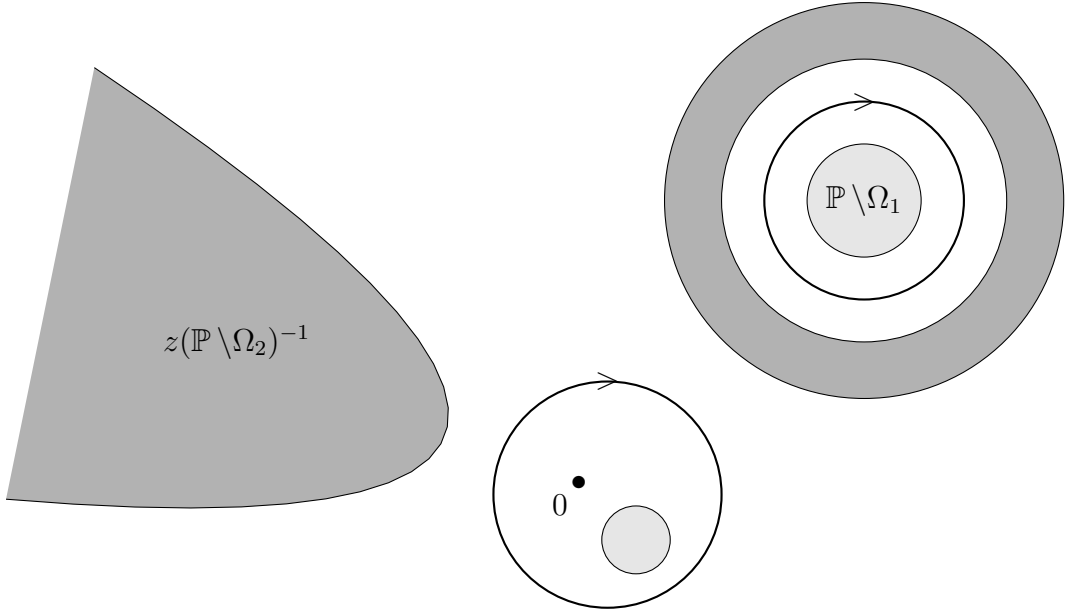

\begin{tikzpicture}[scale=0.75]

\draw(0,0)  node[below left] {$0$} node {$\bullet$};
\draw[fill=gray!20] (1,-1) circle (0.6);
\draw[fill=gray!20] (5,5) circle (1) node {$\PPP\backslash \Omega_1$};
\fill[gray!60,even odd rule] (5,5) circle (2.5) (5,5) circle (3.5);
\draw (5,5) circle (2.5);
\draw (5,5) circle (3.5);
\draw[thick] (0.5,-0.2) circle (2);
\draw[thick] (0.5 , 1.8) node {$>$};
\draw[thick] (5,5) circle (1.75);
\draw[thick] (5 , 6.75) node {$>$};
\draw[fill=gray!60,rotate around={70:(0,0)}, domain = -3.7:4]  plot (\x, 0.5*\x*\x+2.5);
\draw (-6,2.5) node {$z(\PPP\backslash\Omega_2)^{-1}$};
\end{tikzpicture}
\captionof{figure}{A Hadamard cycle for $z(\PPP\backslash\Omega_2)^{-1}$ in $\Omega_1$, in the case where $0,\infty \in \Omega_1$ and $\infty \in \Omega_2, 0\notin \Omega_2.$ }
\end{center}

\begin{proposition}[\cite{Pohl09}, Lemma $3.4.5$ and Proposition $3.6.4$]
The Hadamard product $f_1 \star f_2$ can be continuously extended to $\Omega_1 \star \Omega_2.$ If $0\in \Omega_1 \star \Omega_2$ (resp. $\infty\in \Omega_1 \star \Omega_2$), one has $(f_1 \star f_2)(0) = f_1(0)f_2(0)$ (resp. $(f_1 \star f_2)(\infty) = 0$). Moreover, $f_1 \star f_2$ is an element of $\HH(\Omega_1 \star \Omega_2).$
\end{proposition}

\begin{proposition}[\cite{Pohl09}, Proposition $3.6.1$]
The Hadamard product is commutative.
\end{proposition}

In all this framework, the hypothesis $f(\infty)=0$, when $\infty \in \Omega$, is highly used. In the next section, we shall provide a more general definition of Hadamard cycles and Hadamard product, based on singular homology theory, which does not require the vanishing condition at infinity.

\section{Generalized Hadamard cycles}

For classical facts about singular homology, we refer to \cite{Gree67} and \cite{Hatc02}. For a general background on sheaf theory and derived functors, we refer to \cite{Kash90}. For a sheaf-theoretic definition of the Borel-Moore homology and the link with singular homology on HLC-spaces, we refer to \cite{Bred97}.

\bigskip
Let us recall that on any topological space $X$, there is an orientation complex $\omega_X$ which is canonically isomorphic to $\Z_X[n]$ if $X$ is an oriented topological manifold of pure dimension $n$. On a topological space $X$, the Borel-Moore homology (resp. Borel-Moore homology with compact support) of degree $k$ is defined by

\[
{}^{\BM}\!H_k(X) := H^{-k}(X, \omega_X), \quad {}^{\BM}\!H_k^{c}(X) := H^{-k}_c(X, \omega_X).
\]

\begin{definition}\index{Orientation class}\label{def:OC}
Let $X$ be an oriented topological manifold of pure dimension $n$. \emph{The orientation class} of $X$ is the class

\[
\alpha_X \in {}^{\BM}\!H_n(X) \simeq H^{-n}(X, \Z_X[n]) \simeq H^0(X,\Z_X)
\]
corresponding to the constant section $1$ of $\Z_X$.
\end{definition}

Let $X$ be a topological manifold $X$ of pure dimension $n$. Since $X$ is homologically locally connected, the complex ${\RR}\Gamma_c(X,\omega_X)$ is canonically isomorphic to the complex of singular chains on $X$. Hence, ${}^{\BM}\!H_k^{c}(X)$ is isomorphic to the usual singular homology group of degree $k$, $H_k(X)$. Now, let $K$ be a compact subset of $X$ and consider the two canonical excision distinguished triangles

$$ {\RR}\Gamma_{X\backslash K}(X,\omega_X) \to {\RR}\Gamma(X,\omega_X) \to {\RR}\Gamma(K,\omega_X)\overset{+}\to$$ and $${\RR}\Gamma_{c}(X\backslash K,\omega_X) \to {\RR}\Gamma_c(X,\omega_X) \to {\RR}\Gamma(K,\omega_X) \overset{+}\to.$$ The second triangle implies that $H^{-n}(K,\omega_X)$ is canonically isomorphic to the relative singular homology group $H_{n}(X,X\backslash K)$. Hence, we get a sequence of morphisms $$ {}^{\BM}\!H_n(X) \to H^{-n}(K,\omega_X) \xrightarrow{\sim} H_{n}(X,X\backslash K)$$ and $\alpha_X \in {}^{\BM}\!H_n(X)$ induces a relative orientation class $\alpha_{X, K} \in H_{n}(X,X\backslash K).$

\begin{proposition}\label{prop:homologous}
Let $\Omega$ be a proper open subset of $\C$ and let $F = \C \backslash \Omega.$ There is a canonical isomorphism
\[
H_1(\Omega) \xrightarrow{\sim} H^0_c(F, \Z_F)
\]
given by

\[
[c] \mapsto \left(z\mapsto \Ind_z(c) \right).
\]
\end{proposition}
\begin{proof}
Let us consider the excision distinguished triangle

\begin{equation}\label{equ:triangle}
{\RR}\Gamma_c(\Omega, \omega_{\C}) \to {\RR}\Gamma_c(\C, \omega_{\C}) \to {\RR}\Gamma_c(F,\omega_{\C}) \overset{+1}\to.
\end{equation}

\noindent
It induces a long exact sequence \begin{center}
\begin{tikzcd}
\cdots \rar & H_2(\Omega) \rar & H_2(\C) \rar
             \ar[draw=none]{d}[name=X, anchor=center]{}
    & H^{-2}{\RR}\Gamma_c(F,\omega_{F}) \ar[rounded corners,
            to path={ -- ([xshift=2ex]\tikztostart.east)
                      |- (X.center) \tikztonodes
                      -| ([xshift=-2ex]\tikztotarget.west)
                      -- (\tikztotarget)}]{dll}[at end]{} \\      
  & H_1(\Omega) \rar & H_1(\C) \rar & H^{-1}{\RR}\Gamma_c(F,\omega_{F}) \rar & \cdots
\end{tikzcd}
\end{center}
\noindent
Since $\C$ is contractible, one has $H_2(\C) \simeq H_1(\C) \simeq \{0\}$. Therefore, taking into account that $\omega_{F} \simeq \Z_F[2],$ one gets a canonical isomorphism

\[
\delta : H^0_c(F,\Z_F) \xrightarrow{\sim} H_1(\Omega).
\]
Let $z\in F.$ Applying (\ref{equ:triangle}) with $\C \backslash\{z\}, \C$ and $\{z\},$ one gets an isomorphism

\[
\delta_z : \Z \simeq H^0_c(\{z\},\Z_{\{z\}}) \xrightarrow{\sim} H_1(\C\backslash \{z\}).
\]
Clearly, $\delta_z^{-1}([c]) = \Ind_z(c)$. Moreover, by Proposition $1.3.6$ in \cite{Kash90}, there is a commutative diagram  

\[
\xymatrix{
H^0_c(F,\Z_F)\ar[r]^{\delta}\ar[d]_{i_z} &H_1(\Omega)\ar[d]^{j_z}\\
H^0_c(\{z\},\Z_{\{z\}})\ar[r]_{\delta_z} &H_1(\C\backslash \{z\})
}
\]
where $i_z(f)=f(z)$ and $j_z([c])=[c].$ Hence, one sees that $\delta^{-1}([c])(z) = \Ind_z(c).$ Since this argument is valid for all $z\in F$, the conclusion follows.
\end{proof}

To introduce our definition of generalized Hadamard cycles, we have to be in the same setting as T. Pohlen. However, looking at Definition~\ref{def:stareligible}, we find it more natural to start with closed subsets instead of open ones. 

\begin{definition}\label{def:closedstareligible}\index{Star-eligible closed sets}
Two closed subsets $S_1$ and $S_2$ of $\PPP$ are \emph{star-eligible} if $S_1,S_2$ and $S_1S_2$ are proper and if $S_1 \times S_2 \subset M.$
\end{definition}

For the rest of the section we fix $S_1$ and $S_2$, two star-eligible closed subsets of $\PPP$. If $z\in \C^* \backslash S_1S_2$, $S_1$ is a compact subset of $\PPP\backslash zS_2^{-1}$ and, thus, a compact subset of $\PPP\backslash (zS_2^{-1} \cup (\{0,\infty\}\backslash S_1))$. Moreover, one has 
\[
(\PPP\backslash (zS_2^{-1} \cup (\{0,\infty\}\backslash S_1)) )\backslash S_1 = \PPP\backslash (S_1 \cup zS_2^{-1} \cup \{0\} \cup \{\infty\}).
\]
Let $z\in \C^* \backslash S_1S_2.$

\begin{definition}\label{def:genhadcycle}\index{Generalized Hadamard cycle}
A \emph{generalized Hadamard cycle} for $S_1$ in $\PPP\backslash (zS_2^{-1} \cup (\{0,\infty\}\backslash S_1))$ is a representative $c$ of the class in $H_1(\PPP\backslash (S_1 \cup zS_2^{-1} \cup \{0\} \cup \{\infty\}))$ which is the image of

\[
-\alpha_{\PPP\backslash (zS_2^{-1} \cup (\{0,\infty\}\backslash S_1)),S_1} \in H_2(\PPP\backslash (zS_2^{-1} \cup (\{0,\infty\}\backslash S_1)),\PPP\backslash (S_1 \cup zS_2^{-1}\cup \{0\}\cup \{\infty\}))
\]
by the canonical map

\[
\xymatrix{
H_2(\PPP\backslash (zS_2^{-1} \cup (\{0,\infty\}\backslash S_1)),\PPP\backslash (S_1 \cup zS_2^{-1}\cup \{0\}\cup \{\infty\})) \ar[d] \\ H_1(\PPP\backslash (S_1 \cup zS_2^{-1} \cup \{0\} \cup \{\infty\})).
}
\]
\end{definition}

Our aim is now to define a product

\[
\Hol(\PPP \backslash S_1) \times \Hol(\PPP \backslash S_2) \to \Hol(\C^*\backslash S_1S_2)
\]
which generalizes the extended Hadamard product of T. Pohlen.

\begin{definition}\label{def:genhadproduct}\index{Generalized Hadamard product}
Let $f_1 \in \Hol(\PPP \backslash S_1)$ and $f_2 \in \Hol(\PPP \backslash S_2).$ For each $z\in \C^*\backslash S_1S_2$ we set

\[
(f_1\star f_2)(z) = \frac{1}{2i\pi}\int_{c_z} f_1(\zeta)f_2\left(\frac{z}{\zeta}\right) \frac{d\zeta}{\zeta},
\]
where $c_z$ is a generalized Hadamard cycle for $S_1$ in $\PPP\backslash (zS_2^{-1} \cup (\{0,\infty\}\backslash S_1))$. Since two generalized Hadamard cycles are homologous, the definition does not depend on the chosen generalized Hadamard cycle. The function $f_1 \star f_2$ is called the \emph{generalized Hadamard product} of $f_1$ and $f_2.$
\end{definition}

\begin{center}
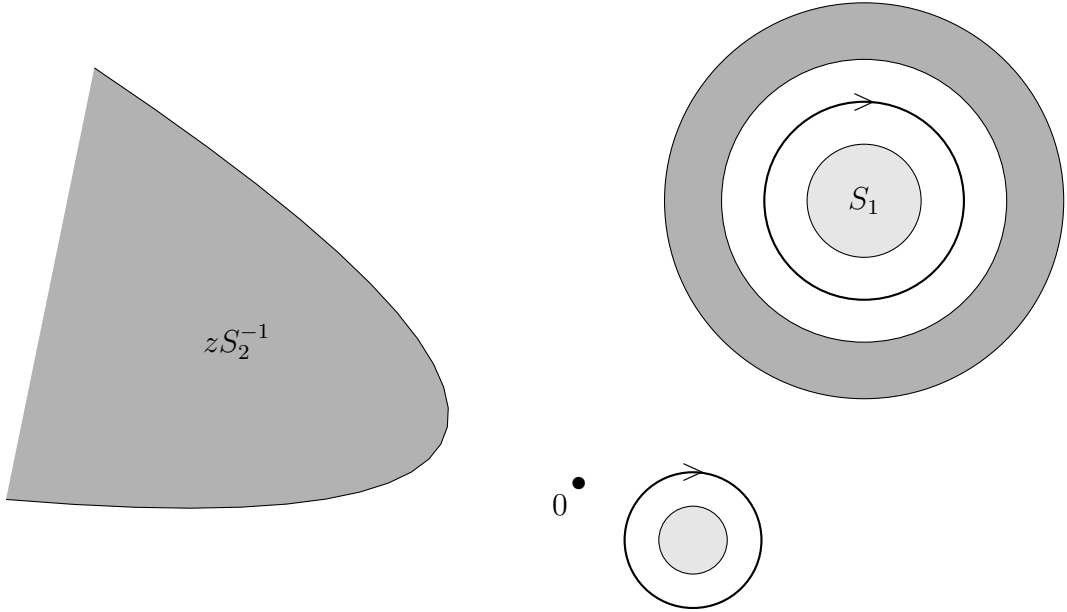

\begin{tikzpicture}[scale=0.75]

\draw(0,0)  node[below left] {$0$} node {$\bullet$};
\draw[fill=gray!20] (2,-1) circle (0.6);
\draw[fill=gray!20] (5,5) circle (1) node {$S_1$};
\fill[gray!60,even odd rule] (5,5) circle (2.5) (5,5) circle (3.5);
\draw (5,5) circle (2.5);
\draw (5,5) circle (3.5);
\draw[thick] (2,-1) circle (1.2);
\draw[thick] (2,0.2) node {$>$};
\draw[thick] (5,5) circle (1.75);
\draw[thick] (5 , 6.75) node {$>$};
\draw[fill=gray!60,rotate around={70:(0,0)}, domain = -3.7:4]  plot (\x, 0.5*\x*\x+2.5);
\draw (-6,2.5) node {$zS_2^{-1}$};
\end{tikzpicture}
\captionof{figure}{A generalized Hadamard cycle for $S_1$ in $\PPP\backslash (zS_2^{-1} \cup (\{0,\infty\}\backslash S_1))$, in the case where $0, \infty \notin S_1$ and $0 \in S_2, \infty\notin S_2.$ }
\end{center}

\begin{lemma}\label{lem:uniform}
Let $f_1 \in \Hol(\PPP \backslash S_1)$ and $f_2 \in \Hol(\PPP \backslash S_2).$ For each compact subset $K$ of $\C^*\backslash S_1S_2$, there is a cycle $c_K$ in $\PPP\backslash(S_1 \cup KS_2^{-1} \cup \{0\} \cup \{\infty\})$ such that

\[
(f_1\star f_2)(z) = \frac{1}{2i\pi}\int_{c_K} f_1(\zeta)f_2\left(\frac{z}{\zeta}\right) \frac{d\zeta}{\zeta},
\]
for all $z\in K.$
\end{lemma}

\begin{proof}
There is a fundamental class

\[
\alpha_{\PPP\backslash (KS_2^{-1} \cup (\{0,\infty\}\backslash S_1)),S_1} \in H_2(\PPP\backslash (KS_2^{-1} \cup (\{0,\infty\}\backslash S_1)),\PPP\backslash (S_1 \cup KS_2^{-1}\cup \{0\}\cup \{\infty\}))
\]
We choose $c_K$ to be a representative of the class in $H_1(\PPP\backslash(S_1 \cup KS_2^{-1} \cup \{0\} \cup \{\infty\}))$ which is the image of $-\alpha_{\PPP\backslash (KS_2^{-1} \cup (\{0,\infty\}\backslash S_1)),S_1}$ by the canonical map

\[
\xymatrix{
H_2(\PPP\backslash (KS_2^{-1} \cup (\{0,\infty\}\backslash S_1)),\PPP\backslash (S_1 \cup KS_2^{-1}\cup \{0\}\cup \{\infty\})) \ar[d] \\ H_1(\PPP\backslash (S_1 \cup KS_2^{-1} \cup \{0\} \cup \{\infty\})).
}
\]
For each $z\in K$, there is a canonical commutative diagram

\[
\xymatrix@C=-3cm{
H_2(\PPP\backslash (KS_2^{-1}\cup (\{0,\infty\}\backslash S_1)),\PPP\backslash (S_1 \cup KS_2^{-1}\!\cup \{0\} \cup \{\infty\})) \ar[rd]\ar[ddd] & \\ 
& H_1(\PPP\backslash (S_1 \cup KS_2^{-1} \cup \{0\} \cup \{\infty\})) \ar[d] \\
& H_1(\PPP\backslash (S_1 \cup zS_2^{-1} \cup \{0\} \cup \{\infty\}))\\
H_2(\PPP\backslash (zS_2^{-1} \cup (\{0,\infty\}\backslash S_1)),\PPP\backslash (S_1 \cup zS_2^{-1}\cup \{0\}\cup \{\infty\})) \ar[ru] & .
}
\]

\medskip
\noindent
Obviously, $\alpha_{\PPP\backslash (zS_2^{-1} \cup (\{0,\infty\}\backslash S_1)),S_1}$ is the image of $\alpha_{\PPP\backslash (KS_2^{-1} \cup (\{0,\infty\}\backslash S_1)),S_1}$ by the left vertical map. Therefore, by the commutativity of the diagram, one can deduce that $c_K$ is a generalized Hadamard cycle for $S_1$ in $\PPP\backslash (zS_2^{-1} \cup (\{0,\infty\}\backslash S_1))$, for all $z\in K$. Hence the conclusion.
\end{proof}

\begin{proposition}
The generalized Hadamard product is a well-defined map

\[
\Hol(\PPP \backslash S_1) \times \Hol(\PPP \backslash S_2) \to \Hol(\C^* \backslash S_1 S_2).
\]
\end{proposition}

\begin{proof}
Let $f_1 \in \Hol(\PPP \backslash S_1)$ and $f_2 \in \Hol(\PPP \backslash S_2).$ We have to check that $f_1\star f_2$ is holomorphic on $\C^* \backslash S_1 S_2$. Since it is a local property, it is enough to prove that $f_1 \star f_2$ is holomorphic on each small open disk $D \subset \C^* \backslash S_1 S_2.$ Let $D$ be such a disk. By Lemma~\ref{lem:uniform} there is a cycle $c_{\overline{D}}$ such that

\[
(f_1\star f_2)(z) = \frac{1}{2i\pi}\int_{c_{\overline{D}}} f_1(\zeta)f_2\left(\frac{z}{\zeta}\right) \frac{d\zeta}{\zeta},
\]
for all $z\in D.$ We conclude by derivation under the integral sign.
\end{proof}

We shall now prove that our product is a good generalization of the extended Hadamard product of T. Pohlen. By doing so, the reader shall see why we chose such a sign convention in Definition~\ref{def:genhadcycle}.

\begin{proposition}
Let $f_1 \in \HH(\PPP\backslash S_1)$ and $f_2 \in \HH(\PPP\backslash S_2)$. Let $z \in \C^* \backslash S_1 S_2.$ Let $c_z$ be a generalized Hadamard cycle for $S_1$ in $\PPP\backslash (zS_2^{-1} \cup (\{0,\infty\}\backslash S_1))$ and $d_z$ a Hadamard cycle for $zS_2^{-1}$ in $\PPP\backslash S_1.$ Then,

\[
\frac{1}{2i\pi}\int_{c_z} f_1(\zeta)f_2\left(\frac{z}{\zeta}\right) \frac{d\zeta}{\zeta} = \frac{1}{2i\pi}\int_{d_z} f_1(\zeta)f_2\left(\frac{z}{\zeta}\right) \frac{d\zeta}{\zeta}.
\]
\end{proposition}
\begin{proof}
We treat the case where $0, \infty \notin S_1$ and $0 \in S_2, \infty\notin S_2$ and leave the other ones to the reader. By construction, it is clear that $c_z$ verifies

\[
\Ind(c_z,w) = \begin{cases} 0\,\, &\text{if} \,\,\,\, w\in zS_2^{-1}\cup \{0\}\\ -1 \,\, &\text{if} \,\,\,\, w\in S_1, \end{cases}
\]
Let $c'_z$ be a cycle $\PPP\backslash(S_1 \cup zS_2^{-1} \cup \{0\} \cup \{\infty\})$ such that 
\[
\Ind(c'_z,w) = \begin{cases} 0\,\, &\text{if} \,\,\,\, w\in zS_2^{-1}\cup S_1\\ -1 \,\, &\text{if} \,\,\,\, w=0. \end{cases}
\]
\noindent
Since $d_z$ is $\text{acc}^-$, by Proposition~\ref{prop:homologous}, it is clear that $d_z$ is homologuous to $c_z+c'_z$ in $\PPP\backslash(S_1 \cup zS_2^{-1} \cup \{0\} \cup \{\infty\})$. We then have

\[
\int_{c_z} f_1(\zeta)f_2\left(\frac{z}{\zeta}\right) \frac{d\zeta}{\zeta} = \int_{d_z} f_1(\zeta)f_2\left(\frac{z}{\zeta}\right) \frac{d\zeta}{\zeta}- \int_{c'_z} f_1(\zeta)f_2\left(\frac{z}{\zeta}\right) \frac{d\zeta}{\zeta}.
\]
Moreover, by the residue theorem,

\begin{align*}
-\int_{c'_z} f_1(\zeta)f_2\left(\frac{z}{\zeta}\right) \frac{d\zeta}{\zeta} & = 2i\pi \text{Res}_{\zeta=0} \left(\frac{f_1(\zeta)}{\zeta}f_2\left(\frac{z}{\zeta}\right)\right) = 2i\pi \lim_{\zeta\to 0} \left(f_1(\zeta)f_2\left(\frac{z}{\zeta}\right)\right)\\
& = 2i\pi f_1(0)f_2(\infty) =0.
\end{align*}

\noindent
Hence the conclusion.
\end{proof}

\begin{remark}
Of course, the generalized Hadamard product is no longer commutative if the functions do not vanish at infinity. For example, let $S_1$ and $S_2$ be as in the proof of the previous proposition. Let $f_1 \in \Hol(\PPP\backslash S_1)$ and $f_2 \in \Hol(\PPP\backslash S_2).$ By a similar computation, one sees that

\[
f_1\star f_2 - f_2 \star f_1 = f_1(0)f_2(\infty).
\]
\end{remark}

Despite the lack of commutativity, the generalized Hadamard cycles are more symmetric with respect to $0$ and $\infty$. In the section $5$, we shall explain how one can define a convolution between $1$-forms which have (not necessarily isolated) singularities at $0$ and $\infty$. Generalized Hadamard cycles are key ingredients to compute such a convolution (see also section $6$). Moreover, the commutativity shall eventually be obtained thanks to quotient spaces that naturally occur in this context.  

\section{The holomorphic integration map}

Let $X$ be a complex manifold of complex dimension $d_X$ and $r \in \Z$. Recall that $\CC_{\infty,X}^r$ admits a decomposition in bi-types
\[
\CC_{\infty,X}^r \simeq \bigoplus_{p+q=r} \CC_{\infty,X}^{p,q}
\]
which induces a decomposition of the exterior derivative $d$ as
\[
d = \partial + \overline{\partial},
\]
where
\[
\partial: \CC_{\infty,X}^{p,q} \to \CC_{\infty,X}^{p+1,q}
\quad\text{and}\quad
\overline{\partial}: \CC_{\infty,X}^{p,q} \to \CC_{\infty,X}^{p,q+1}.
\]
Similarly, $\Db_X^{r}$ admits a decomposition in bi-types
\[
\Db_{X}^r \simeq \bigoplus_{p+q=r} \Db_{X}^{p,q}
\]
and an associated decomposition of the distributional exterior derivative. Moreover, for any open subset $U$ of $X$, we have a canonical isomorphism
\[
\Db_X^{r}(U) \simeq \Gamma_c(U,\CC_{\infty,X}^{2d_X-r})'
\]
between the space of complex distributional $r$-forms and the topological dual of the space of infinitely differentiable complex differential $(2d_X-r)$-forms with compact support which induces the similar isomorphism
\[
   \Db_X^{p,q}(U) \simeq \Gamma_c(U,\CC_{\infty,X}^{d_X-p,d_X-q})'.
\]
In the sequel, we denote by $\Omega_{X}^p$ the sheaf of holomorphic differential $p$-forms on $X$ and we set for short $\Omega_X = \Omega_X^{d_X}.$ Of course, $\Omega_{X}^p$ is canonically isomorphic to both the kernel of
\[
  \overline\partial : \CC_{\infty,X}^{p,0} \to \CC_{\infty,X}^{p,1}
\]
and the kernel of
\[
  \overline\partial :\Db_{X}^{p,0} \to \Db_{X}^{p,1}.
\]

The double complex $\CC_{\infty,X}^{\bigcdot,\bigcdot}$ (resp.\ $\Db_X^{\bigcdot,\bigcdot}$) is the infinitely differentiable (resp.\ distributional) Dolbeault complex of $X$. By construction, the associated simple complex is the infinitely differentiable (resp.\ distributional) de~Rham complex $\CC_{\infty,X}^{\bigcdot}$ (resp.\ $\Db_{X}^{\bigcdot}$) of $X$. Moreover, we have the following chains of canonical quasi-isomorphisms :
\[
   \C_{X}\simeq\CC_{\infty,X}^{\bigcdot}\simeq\Db_{X}^{\bigcdot}
   \quad\text{and}\quad
   \Omega_{X}^{p}\simeq\CC_{\infty,X}^{p,\bigcdot}\simeq\Db_{X}^{p,\bigcdot},
\]
which are given by de Rham and Dolbeault lemmas.

\bigskip

Let $f : X \to Y$ be a holomorphic map from $X$ to a complex manifold $Y$ of complex dimension $d_Y$ and let $V$ be an arbitrary open subset of $Y$. It  follows from the holomorphy of $f$ that the pullback
\[
   f^*:\CC_{\infty,Y}^{r}(V) \to \CC_{\infty,X}^{r}(f^{-1}(V))
\]
sends $\CC_{\infty,Y}^{p,q}(V)$ into $\CC_{\infty,X}^{p,q}(f^{-1}(V))$ if $p+q=r$. In particular,
\[
\partial(f^*\omega) = f^*(\partial \omega)
\quad\text{and}\quad
\overline{\partial}(f^*\omega) = f^*(\overline{\partial} \omega)
\]
for all $\omega\in\CC_{\infty,Y}^{p,q}(V)$. By topological duality, it follows that there are canonical pushforward morphisms
\[
   \int_f : \Gamma_{f-\text{proper}}(f^{-1}(V),\Db_{Y}^{2d_Y-r}) \to \Gamma(V,\Db_{Y}^{2d_X-r})
\]
and
\[
   \int_f :\Gamma_{f-\text{proper}}(f^{-1}(V),\Db_{Y}^{d_Y-p,d_Y-q}) \to \Gamma(V,\Db_{Y}^{d_X-p,d_X-q})
\]
between distributional forms with $f$-proper support on $f^{-1}(V)$ and distributional forms on $V$ and that these morphisms commute with $\partial$ and $\overline{\partial}$. In particular, we get a morphism of double complexes of sheaves of the form
\[
\int_f : f_!\Db_X^{\bigcdot+d_X,\bigcdot+d_X} \to \Db_Y^{\bigcdot+d_Y,\bigcdot+d_Y}.
\]
Moreover, if $f$ is a surjective submersion, one can show that the pushforward of a distibutional form associated to an infinitely differentiable form with $f$-proper support is itself associated to an infinitely differentiable form which can be computed by integration over the fibers of $f$. This shows that, in this case, the preceding morphism factors through a morphism of the form
\[
\int_f : f_!\CC_{\infty,X}^{\bigcdot+d_X,\bigcdot+d_X} \to \CC_{\infty,Y}^{\bigcdot+d_Y,\bigcdot+d_Y}.
\]

Thanks to the quasi-isomorphisms
\[
\Omega_{X}^{p+d_X}\simeq\Db_{X}^{p+d_X,\bigcdot}
\quad\text{and}\quad
\Omega_{Y}^{p+d_Y}\simeq\Db_{Y}^{p+d_Y,\bigcdot},
\]
this gives us a morphism
\[
\int_f : {\RR}f_!\Omega_{X}^{p+d_X}[d_X] \to \Omega_{Y}^{p+d_Y}[d_Y]
\]
in the derived category for each $p\in \Z.$ In the particular case where $p=0$, we get the morphism
\[
\int_f : {\RR}f_!\Omega_{X}[d_X] \to \Omega_{Y}[d_Y]
\]
which is usually called \emph{the holomorphic integration map along the fibers of $f$} (see e.g. \cite[p.\ 129]{Kash90}). Note that, if $g : Y \to Z$ is another holomorphic map between complex manifolds, then the well known relation $(g\circ f)^*=f^*\circ g^*$ entails that $\int_{g \circ f} =\int_g \circ \int_f$.

\section{Holomorphic cohomological convolution}

\label{hccclg}

\begin{definition}\label{def:convolvability}
Let $(G,\mu)$ be a locally compact complex Lie group of complex dimension $n$. Two closed subsets $S_1$ and $S_2$ of $G$ are said to be \emph{convolvable} if $S_1 \times S_2$ is $\mu$-proper, i.e. if
\[
   (S_1 \times S_2) \cap \mu^{-1}(K)
\]
is a compact subset of $G\times G$ for any compact subset $K$ of $G$. 
\end{definition}
\begin{remark}
A proper map on a locally compact topological space is universally closed, in particular closed (see e.g. \cite{Bour66}). Hence, if $S_1$ and $S_2$ are convolvable closed subsets of $G$, then $\mu|_{S_1 \times S_2}$ is a proper map and $S_1+S_2 = \mu|_{S_1 \times S_2}(S_1 \times S_2)$ is closed.
\end{remark}

\begin{definition}\index{Convolvable distributional forms}\index{Convolution of distributional forms}\label{def:convdist}
Two distributional $2n$-forms $u_1$ and $u_2$ of $G$ are \emph{convolvable} if the support $S_1$ of $u_1$ and the support $S_2$ of $u_2$ are convolvable. In that case, the convolution product of $u_1$ and $u_2$ is a distributional $2n$-form on $G$ defined by $$u_1 \star u_2 = \int_{\mu}(u_1 \boxtimes u_2) := \int_{\mu}(p_1^*u_1 \wedge p_2^*u_2),$$ where $p_1, p_2 : G \times G \to G$ are the two canonical projections.
\end{definition}

\begin{remark}
By choosing a Haar form $\nu$ on $G$, one can define the convolution product of two distributions by means of the isomorphism $\Db_{G} \simeq \Db_G^{2n}$ given by $\nu$ (see e.g. \cite{Dieu80}).
\end{remark}

\begin{remark}
If we define
\[
\phi : G\times G \to G \times G \quad\text{and}\quad \psi:G\times G \to G \times G
\]
by setting $\phi(g_1,g_2)=(g_1,\mu(g_1,g_2))$ and $\psi(g_1,g_2)=(g_1,\mu(g_1^{-1}, g_2))$, we see that $\phi$ and $\psi$ are reciprocal biholomorphic bijections and that the diagram
\[
\xymatrix{
G\times G \ar@{->}[rr]_{\phi}^{\sim} \ar@{->}[rd]_{\mu} &   &  G\times G\ar@{->}[ld]^{p_2} \\
                                                          & G
}
\]
is commutative. This shows in particular that $\mu$ is a surjective submersion and that the preceding procedure allows us also to define the convolution product of infinitely differentiable forms.
\end{remark}

Let $S_1$ and $S_2$ be two convolvable closed subsets of $G$. By construction, the convolution of distributions on $G$ is the composition of the external product of distributions
\[
   \Gamma_{S_1}(G,\Db_G^{2n})\otimes\Gamma_{S_2}(G,\Db_G^{2n}) \to  \Gamma_{S_1\times S_2}(G\times G,\Db_{G\times G}^{4n})
\]
and the map
\[
   \int_{\mu} : \Gamma_{S_1\times S_2}(G\times G,\Db_{G\times G}^{4n}) \to \Gamma_{\mu(S_1 \times S_2)}(G,\Db_G^{2n})
\]
induced by the integration map along the fibers of $\mu$
\[
   \int_{\mu} : \Gamma_{\mu-\text{proper}}(G\times G,\Db_{G\times G}^{4n})\to \Gamma(G,\Db_G^{2n})
\]
and the fact that $S_1$ and $S_2$ are convolvable. It is thus natural to define the convolution of cohomology classes of holomorphic forms on $G$ as follows :

\begin{definition}\index{Holomorphic cohomological convolution}\label{def:holcohconv}
Let $S_1, S_2$ be two convolvable closed subsets of $G$. Consider the external product morphisms
\[
{\RR}\Gamma_{S_1}(G,\Omega_G^{p+n})[n] \otimes {\RR}\Gamma_{S_2}(G,\Omega_G^{q+n})[n] \to {\RR}\Gamma_{S_1 \times S_2}(G \times G,\Omega_{G\times G}^{p+q+2n})[2n]
\]
and the morphisms
\[
\int_{\mu} : {\RR}\Gamma_{S_1 \times S_2}(G \times G,\Omega_{G\times G}^{p+q+2n})[2n] \to {\RR}\Gamma_{\mu(S_1\times S_2)}(G,\Omega_{G}^{p+q+n})[n].
\]
induced by the holomorphic integration map and the fact that $S_1\times S_2$ is $\mu$-proper. By composition, these morphisms
give derived category morphisms
\[
\star_{(G,\mu)} : {\RR}\Gamma_{S_1}(G,\Omega_G^{p+n})[n] \otimes {\RR}\Gamma_{S_2}(G,\Omega_G^{q+n})[n] \rightarrow {\RR}\Gamma_{\mu(S_1\times S_2)}(G,\Omega_{G}^{p+q+n})[n],
\]
that we call the \emph{holomorphic convolution morphisms of $G$}. Going to cohomology groups, these morphisms give rise to the morphisms
\[
\star_{(G,\mu)} : H_{S_1}^{r+n}(G,\Omega_G^{p+n}) \otimes H_{S_2}^{s+n}(G,\Omega_G^{q+n}) \rightarrow H_{\mu(S_1 \times S_2)}^{r+s+n}(G,\Omega_{G}^{p+q+n}),
\]
that we call the \emph{holomorphic cohomological convolution morphisms of $G$}.
\end{definition}

\begin{remark}
Consider the diagram 
\[
\hspace*{-1.35cm}
\xymatrix@=3em{
& H_{S_1}^{n}(G,\Omega_G) \otimes H_{S_2}^{n}(G,\Omega_G) \ar[r] & H_{\mu(S_1 \times S_2)}^{n}(G,\Omega_{G})\\
& \Gamma_{S_1}(G,\Db_G^{2n})\otimes\Gamma_{S_2}(G,\Db_G^{2n}) \ar[r]\ar[u] &\Gamma_{\mu(S_1 \times S_2)}(G,\Db_G^{2n}) \ar[u] }
\]
where the vertical arrows are given by the Dolbeault complex of $\Omega_G$ and the top (resp the bottom) horizontal arrow is given by the holomorphic cohomological morphism of $G$ with $p=q=r=s=0$ (resp. the convolution product of distributions). Obviously, by the definitions, this diagram is commutative. This remark will allow to perform explicit computations in the next section. 
\end{remark}

\section{Multiplicative convolution on $\C^*$}

In this section, we will consider the case where the group $G$ is the group $\C^*$ formed by the set of non-zero complex numbers endowed with the complex multiplication (noted as a concatenation). We will assume that $S_1, S_2$ are convolvable proper closed subsets of $\C^*$ (remark that this means that $S_1 \cap KS_2^{-1}$ is compact for any compact subset $K$ of $\C^*$) such that $S_1S_2$ is also a proper subset of $\C^*$ and we will show how to compute the holomorphic cohomological convolution morphism
\begin{equation}\label{equ:holconv}
\star : H_{S_1}^{1}(\C^*,\Omega_{\C^*}) \otimes H_{S_2}^{1}(\C^*,\Omega_{\C^*}) \rightarrow H_{S_1 S_2}^{1}(\C^*,\Omega_{\C^*})
\end{equation}
by means of path integral formulas. 

\begin{proposition}
Let $S$ be a proper closed subset of $\C^*$, then there is a canonical isomorphism
\[
        H^r_S(\C^*,\Omega_{\C^*}) \simeq
        \begin{cases}
                \Omega(\C^*\setminus S)/\Omega(\C^*) & \text{if }r=1,\\
                0                                            & \text{otherwise.}
        \end{cases}
\]
\end{proposition}
\begin{proof}
Any open subset of $\C$ is a Stein manifold.
\end{proof}

Thanks to this proposition, one can see that (\ref{equ:holconv}) can be interpreted as a bilinear map $$\star : \Omega(\C^*\backslash S_1)/\Omega(\C^*) \times \Omega(\C^*\backslash S_2)/\Omega(\C^*) \to \Omega(\C^*\backslash S_1S_2)/\Omega(\C^*).$$ Now, let $\omega_1\in\Omega(\C^*\setminus S_1)$ and $\omega_2\in\Omega(\C^*\setminus S_2)$ be two given holomorphic forms. Ideally, we would like to obtain a formula of the form
\[
        [\omega_1] \star [\omega_2] = [\omega]
\]
where $\omega$ is a holomorphic form on $\C^*\setminus (S_1S_2)$ which can be computed from $\omega_1$ and $\omega_2$ by some path integral.

\bigskip

It is in general not possible to find such a nice formula. However, we will show that for any relatively compact open subset $U$ of $\C^*$ and any open neighbourhood $V$ of $S_1S_2$ in $\C^*$, there is a holomorphic form $\omega$ on $U\setminus \overline{V}$ which can be computed from $\omega_1$ and $\omega_2$ by some path integral and which is such that
\[
        [\omega]\in \Omega(U \backslash \overline{V})/\Omega(U) \simeq H^1_{\overline{V}\cap U}(U,\Omega_{\C^*})
\]
coincides with the image of $[\omega_1] \star [\omega_2]$ by the canonical restriction morphism
\[
        H_{S_1 S_2}^{1}(\C^*,\Omega_{\C^*})\to H^1_{\overline{V}\cap U}(U,\Omega_{\C^*}).
\]
Thanks to the following lemma, this is in fact sufficient to completely compute $[\omega_1] \star [\omega_2]$.

\begin{lemma}
Let $S$ be a closed subset of $\C^*$. Then
\[
        H^1_S(\C^*,\Omega_{\C^*}) \simeq \varprojlim_{U\in\mathcal{U}_{\text{rc}}, V\in\mathcal{V}_S} H^1_{\overline{V}\cap U}(U,\Omega_{\C^*})
\]
where $\mathcal{U}_{\text{rc}}$ denotes the set of relatively compact open subsets of $\C^*$ ordered by $\subset$ and $\mathcal{V}_S$ denotes the set of open neighbourhoods of $S$ in $\C^*$ ordered by $\supset$.
\end{lemma}

\begin{proof}
This follows from the Mittag-Leffler theorem for projective systems (see e.g. Proposition $2.7.1$ in \cite{Kash90}).
\end{proof}

To be able to specify the kind of path integral we need, let us first introduce the following definition :

\begin{definition}\index{Relative Hadamard cycle}
Let $F$ and $G$ be two closed subsets of $\C^*$ which have a compact intersection and let $W$ be an open neighbourhood of $F\cap G$.
A \emph{relative Hadamard cycle for $F$ with respect to $G$ in $W$} is a relative 1-cycle
\[
        c\in Z_1(W\setminus F,(W\setminus F)\cap(W\setminus G))
\]
such that its class
\[
        [c]\in H_1(W\setminus F,(W\setminus F)\cap(W\setminus G))
\]
is the image of the fundamental class
\[
        \alpha_{W, F \cap G} \in H_2(W,W\setminus (F\cap G))
\]
by the Mayer-Vietoris morphism
\[
        H_2(W,W\setminus (F\cap G)) \to H_1(W\setminus F,(W\setminus F)\cap(W\setminus G))
\]
associated with the decomposition
\[
        (W,W\setminus (F\cap G)) = ((W\setminus F)\cup W,(W\setminus F)\cup(W\setminus G)).
\]
\end{definition}

\begin{center}
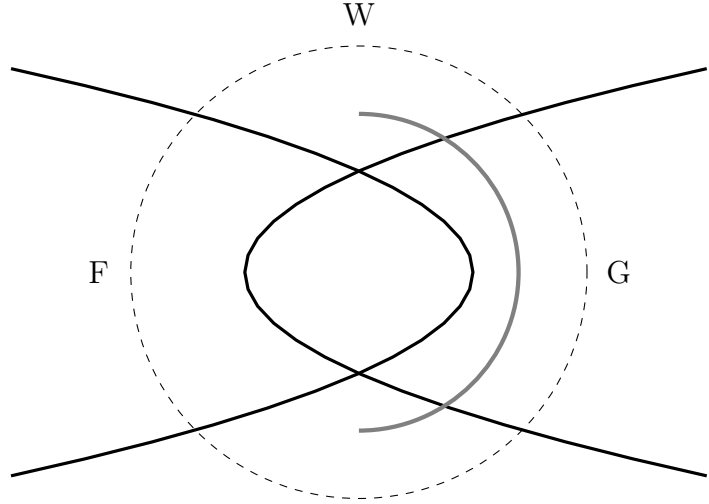

\begin{tikzpicture}[scale=0.60]

\draw[dashed] (0,0) circle (5);
\draw[very thick, rotate around={-90:(0,0)}, domain = -4.5:4.5]  plot (\x, 0.5*\x*\x-2.5);
\draw[very thick, rotate around={-90:(0,0)}, domain = -4.5:4.5]  plot (\x, -0.5*\x*\x+2.5);
\draw[ultra thick,gray] (0,-3.5) arc (-90:90:3.5);
\draw (0,5.7) node [scale=1] {W};
\draw (5.7,0) node [scale=1] {G};
\draw (-5.7,0) node [scale=1] {F};
\end{tikzpicture}
\captionof{figure}{In grey, a relative Hadamard cycle for $F$ with respect to $G$ in $W$.}
\end{center}

\bigskip

With this definition at hand, we can now state the main result of this section.

\begin{theorem}\label{thm:cstarconvol}
Let $S_1$ and $S_2$ be two convolvable proper closed subsets of $\C^*$ such that $S_1S_2\neq\C^*$ and let us assume that $\omega_1=f_1 dz$ and $\omega_2=f_2 dz$ with $f_1\in\mathcal{O}(\C^*\setminus S_1), f_2\in\mathcal{O}(\C^*\setminus S_2)$. Fix a relatively compact open subset $U$ of $\C^*$ and an open neighbourhood $V$ of $S_1S_2$ in $\C^*$. Then, the image of
\[
        [\omega_1]\star[\omega_2] \in \Omega(\C^* \backslash S_1 S_2)/\Omega(\C^*) \simeq H_{S_1 S_2}^{1}(\C^*,\Omega_{\C^*})
\]
in
\[
         \Omega(U\setminus \overline{V})/\Omega(U) \simeq H^1_{\overline{V}\cap U}(U,\Omega_{\C^*})
\]
is the class of the form $\omega=fdz\in\Omega(U\setminus \overline{V})$ where
\[
        f(z) = \int_c f_1(\zeta) f_2\left(\frac{z}{\zeta}\right) \frac{d\zeta}{\zeta}
\]
and $c$ is a relative Hadamard cycle for $S_1$ with respect to $\overline{U}S_2^{-1}$ in $\C^*\setminus (\overline{U}\setminus V)S_2^{-1}$.
\end{theorem}

\begin{lemma}\label{lem:convolnbh}
Let $S_1$ and $S_2$ be convolvable closed subsets of $\C^*$ and let $\mathcal{W}$ be a fundamental system of compact neighbourhoods of $1$ in $\C^*$.
Then
\begin{enumerate}
\item The set $S_1^W=W S_1$ (resp.\ $S_2^W=W S_2$, $S_1^W S_2^W=W^2 S_1 S_2$) is a closed neighbourhood of $S_1$ (resp.\ $S_2$, $S_1 S_2$) in $\C^*$ for any $W\in\mathcal{W}$.
\item The closed subsets $S_1^W$ et $S_2^W$ are convolvable in $\C^*$ for any $W\in\mathcal{W}$.
\item One has $\bigcap_{W\in\mathcal{W}}S_1^W=S_1$, $\bigcap_{W\in\mathcal{W}}S_2^W=S_2$, and $\bigcap_{W\in\mathcal{W}}S_1^WS_2^W=S_1S_2$.
\item
In particular, if $S_1$ and $S_2$ are proper convolvable closed subsets of $\C^*$ such that $S_1S_2\neq\C^*$, if $U$ is a relatively compact open subset of $\C^*$ and if $V$ is an open neighbourhood of $S_1S_2$ in $\C^*$, then there is $W\in\mathcal{W}$ such that $S_1^W$ and $S_2^W$ are convolvable proper closed subsets of $\C^*$ such that $S_1^W S_2^W \neq\C^*$ and $S_1^W S_2^W\cap\overline{U}\subset V$.
\end{enumerate}
\end{lemma}

\begin{proof}
(1) This follows from the fact that $F K$ is closed in $\C^*$ if $F$ (resp.\ $K$) is closed (resp.\ compact) in $\C^*$ and from the fact that
$(zW)_{W\in\mathcal{W}}$ is a fundamental system of neighbourhoods of $z\in\C^*$.

\noindent
(2) This follows from the inclusion
\[
        S_1^W\cap K (S_2^W)^{-1}
        =
        W S_1 \cap K W^{-1}S_2^{-1}
        \subset
        W (S_1 \cap K W^{-2}S_2^{-1})
\]
which is satisfied for any compact subset $K$ of $\C^*$.

\noindent
(3) This is clear since for any closed subset $F$ of $\C^*$ and any $z\not\in F$ there is $W\in\mathcal{W}$ such that $zW^{-1}\cap F=\emptyset$.

\noindent
(4) By contradiction, assume that $$S_1^WS_2^W \cap \overline{U} \cap (\C^* \setminus V) \neq \emptyset$$ for all $W \in \mathcal{W}.$ Then, by compactness, $$\bigcap_{W\in \mathcal{W}}( S_1^WS_2^W \cap \overline{U} \cap (\C^* \setminus V)) = S_1S_2 \cap \overline{U} \cap (\C^* \setminus V) \neq \emptyset,$$ but this contradicts the fact that $S_1^W S_2^W\cap\overline{U}\subset V$. 
\end{proof}

\begin{lemma}\label{lem:representation}
Let $S$ be a proper closed subset of $\C^*$ and let $\omega\in\Omega(\C^*\setminus S)$. Assume that $\omega$ admits an infinitely differentiable extension to $\C^*$ and denote by $\underline{\omega}$ such an extension. Then $[\omega]$, seen as an element of $H^1_S(\C^*,\Omega_{\C^*})$, is the image of
\[
[\overline{\partial}\underline{\omega}]\in H^1(\Gamma_S(\C^*,\CC_{\infty,\C^*}^{1,\bigcdot}))
\]
by the canonical morphism obtained by applying $H^1$ to the composition in the derived category of the canonical morphism
\[
\Gamma_S(\C^*,\CC_{\infty,\C^*}^{1,\bigcdot}) \to {\RR}\Gamma_S(\C^*,\CC_{\infty,\C^*}^{1,\bigcdot})
\]
and the inverse of the canonical isomorphism
\[
{\RR}\Gamma_S(\C^*,\Omega_{\C^*})\xrightarrow{\sim} {\RR}\Gamma_S(\C^*,\CC_{\infty,\C^*}^{1,\bigcdot}).
\]
\end{lemma}

\begin{proof}
It follows from the distinguished triangle
\[
{\RR}\Gamma_S(\C^*,\Omega_{\C^*})
\rightarrow {\RR}\Gamma(\C^*,\Omega_{\C^*})
\rightarrow {\RR}\Gamma(\C^*\setminus S,\Omega_{\C^*})
\overset{+1}\rightarrow
\]
that ${\RR}\Gamma_S(\C^*,\Omega_{\C^*})$ is canonically isomorphic to the mapping cone $M(\rho_S)$ of the restriction morphism
\[
\rho_S : \CC_{\infty,\C^*}^{1,\bigcdot}(\C^*) \to \CC_{\infty,\C^*}^{1,\bigcdot}(\C^*\setminus S)
\]
shifted by $-1$. We know that $M[\rho_S][-1]$ is a complex concentrated in degrees $0$, $1$ and $2$ of the form
\[
\CC_{\infty,\C^*}^{1,0}(\C^*)
\to
\CC_{\infty,\C^*}^{1,1}(\C^*) \oplus \CC_{\infty,\C^*}^{1,0}(\C^*\setminus S)
\to
\CC_{\infty,\C^*}^{1,1}(\C^*\setminus S)
\]
where the differentials in degree $0$ and $1$ are given by the matrices
\[
\begin{pmatrix}
\overline{\partial}\\
-\rho_S        
\end{pmatrix}
\quad\text{and}\quad
\begin{pmatrix}
-\rho_S        & -\overline{\partial}
\end{pmatrix}
\]
What we have to show is that
\[
\begin{pmatrix}
\overline{\partial}\underline{\omega}\\
0
\end{pmatrix}
\quad\text{and}\quad
\begin{pmatrix}
0\\
\omega
\end{pmatrix}
\] are two 1-cycles of this complex which are in the same cohomology class. This is clear since
\[
\begin{pmatrix}
\overline{\partial}\\
-\rho_S        
\end{pmatrix}\underline{\omega}
+
\begin{pmatrix}
0\\
\omega
\end{pmatrix}
=
\begin{pmatrix}
\overline{\partial}\underline{\omega}\\
0
\end{pmatrix}.
\]
\end{proof}

\begin{proof}[Proof of Theorem~\ref{thm:cstarconvol}]
Let $U$ and $V$ be as in the statement of the theorem. Thanks to Lemma~\ref{lem:convolnbh}, we know that it is possible to find a closed neighbourhood $\underline{S}_1$ of $S_1$ and a closed neighbourhood $\underline{S}_2$ of $S_2$ in $\C^*$ such that $\underline{S}_1$ and $\underline{S}_2$ are convolvable and
\[
        \overline{U}\cap\underline{S}_1\underline{S}_2\subset V.
\]
Let $\underline{f}_1$ (resp.\ $\underline{f}_2$) be an infinitely differentiable function on $\C^*$ which coincides with $f_1$ (resp.\ $f_2$) on $\C^*\setminus \underline{S}_1$ (resp.\ $\C^*\setminus \underline{S}_2$) and set
\[
        \underline{\omega}_1 = \underline{f}_1(z)\,dz\quad\text{and}\quad\underline{\omega}_2 = \underline{f}_2(z)\,dz.
\]
It follows from Lemma~\ref{lem:representation} that the image of
\[
        [\omega_1] \in \Omega(\C^*\backslash S_1)/\Omega(\C^*) \simeq H^1_{S_1}(\C^*,\Omega_{\C^*})
\]
by the canonical morphism
\[
        H^1_{S_1}(\C^*,\Omega_{\C^*}) \to H^1_{\underline{S}_1}(\C^*,\Omega_{\C^*})
\]
is the same as the image of
\[
        [\overline{\partial}\underline{\omega}_1] \in H^1(\Gamma_{\underline{S}_1}(\C^*,\mathcal{C}^{(1,\bigcdot)}_{\infty,\C^*}))
\]
by the canonical morphism
\[
         H^1(\Gamma_{\underline{S}_1}(\C^*,\mathcal{C}^{(1,\bigcdot)}_{\infty,\C^*}))
         \to
         H^1_{\underline{S}_1}(\C^*,\Omega_{\C^*})
\]
considered in this lemma. A similar conclusion is true for the image of $$[\omega_2] \in \Omega(\C^*\backslash S_2)/\Omega(\C^*) \simeq H^1_{S_2}(\C^*,\Omega_{\C^*})$$ in $ H^1_{\underline{S}_2}(\C^*,\Omega_{\C^*})$. Therefore, the image of
\[
        [\omega_1]\star[\omega_2] \in \Omega(\C^*\backslash S_1 S_2)/\Omega(\C^*) \simeq H^1_{S_1S_2}(\C^*,\Omega_{\C^*})
\]
in $H^1_{\underline{S}_1\underline{S}_2}(\C^*,\Omega_{\C^*})$ is the same as the image of $[\overline{\partial}\underline{\omega}_1\star\overline{\partial}\underline{\omega}_2]$ by the canonical morphism
\[
         H^1(\Gamma_{\underline{S}_1\underline{S}_2}(\C^*,\mathcal{C}^{(1,\bigcdot)}_{\infty,\C^*}))
         \to
         H^1_{\underline{S}_1\underline{S}_2}(\C^*,\Omega_{\C^*}).
\]

\bigskip

Let us note $p_1, p_2 : \C^* \times \C^* \to \C^*$ the two canonical projections and consider the commutative diagram
\[
\xymatrix{
\C^*\times \C^* \ar@<.5ex>[rr]^{\phi} \ar@{->}[rd]_{\mu} &      &  \C^*\times\C^* \ar@{->}[ld]^{p_2} \ar@<.5ex>[ll]^{\psi} \\
                                                    & \C^* &
}
\]
where $\phi(z_1,z_2)=(z_1,z_1 z_2)$ and $\psi(\zeta,z)=(\zeta,z/\zeta)$. Since $\phi\circ\psi=\text{id}=\psi\circ\phi$, we have
\[
        \int_{\mu}=\int_{p_2}\circ\int_{\phi}=\int_{p_2}\circ\:\psi^*.
\]
Therefore,
\begin{align*}
\overline{\partial}\underline{\omega}_1\star\overline{\partial}\underline{\omega}_2
&= \int_{\mu}(\overline{\partial}\underline{\omega}_1\boxtimes\overline{\partial}\underline{\omega}_2)\\
&= \int_{p_2}(\psi^*(p_1^*\overline{\partial}\underline{\omega}_1\wedge p_2^*\overline{\partial}\underline{\omega}_2))\\
&= \int_{p_2}(p_1^*\overline{\partial}\underline{\omega}_1\wedge h^*\overline{\partial}\underline{\omega}_2)),
\end{align*}
where $h(\zeta,z)=z/\zeta$. Since
\[
        \overline{\partial}\underline{\omega}_1 = \frac{\partial\underline{f}_1}{\partial\overline{z}}(z)d\overline{z}\wedge dz
        \quad\text{and}\quad
        \overline{\partial}\underline{\omega}_2 = \frac{\partial\underline{f}_2}{\partial\overline{z}}(z)d\overline{z}\wedge dz,
\]
we have
\begin{align*}
h^*\overline{\partial}\underline{\omega}_2
&=\frac{\partial\underline{f}_2}{\partial\overline{z}}\left(\frac{z}{\zeta}\right)
  d\left(\frac{\overline{z}}{\overline{\zeta}}\right)\wedge d\left(\frac{z}{\zeta}\right)\\
&=\frac{\partial\underline{f}_2}{\partial\overline{z}}\left(\frac{z}{\zeta}\right)
  \frac{\overline{\zeta} d\overline{z}-\overline{z} d\overline{\zeta}}{\overline{\zeta}^2} \wedge
  \frac{\zeta dz-z d\zeta}{\zeta^2}
\end{align*}
and
\[
        p_1^*\overline{\partial}\underline{\omega}_1\wedge h^*\overline{\partial}\underline{\omega}_2
        = \frac{\partial\underline{f}_1}{\partial\overline{z}}(\zeta)\frac{\partial\underline{f}_2}{\partial\overline{z}}\left(\frac{z}{\zeta}\right)
          \frac{d\overline{\zeta}}{\overline{\zeta}}\wedge\frac{d\zeta}{\zeta}\wedge d\overline{z}\wedge dz.
\]
Therefore,
\[
        \overline{\partial}\underline{\omega}_1 \star \overline{\partial}\underline{\omega}_2
        = \left(\int_{\C^*}
                \frac{\partial\underline{f}_1}{\partial\overline{z}}(\zeta)\frac{\partial\underline{f}_2}{\partial\overline{z}}\left(\frac{z}{\zeta}\right)
                \frac{d\overline{\zeta}}{\overline{\zeta}}\wedge\frac{d\zeta}{\zeta}
          \right) d\overline{z}\wedge dz.
\]
Since $\underline{f}_1$ coincides with $f_1$ on $\C^*\setminus\underline{S}_1$, one has
\[
        \supp\left(\zeta\mapsto\frac{\partial\underline{f}_1}{\partial\overline{z}}(\zeta)\right)
        \subset
        \underline{S}_1.
\]
Similarly, one has
\[
        \supp\left(\zeta\mapsto\frac{\partial\underline{f}_2}{\partial\overline{z}}\left(\frac{z}{\zeta}\right)\right)
        \subset
        z\underline{S}_2^{-1}.
\]
Hence,
\[
        \zeta
        \mapsto
        \frac{\partial\underline{f}_1}{\partial\overline{z}}(\zeta)\frac{\partial\underline{f}_2}{\partial\overline{z}}\left(\frac{z}{\zeta}\right)
\]
is an infinitely differentiable function on $\C^*$ supported by $\underline{S}_1\cap z\underline{S}_2^{-1}$ which is a compact subset of $\C^*$.

\bigskip

Since $U$ is a relatively compact open subset of $\C^*$ and $\underline{S}_1$ and $\underline{S}_2$ are convolvable closed subsets of $\C^*$,
\[
        K=\underline{S}_1\cap\overline{U}\underline{S}_2^{-1}
\]
is a compact subset of $\C^*$. Let $c$ be a singular infinitely differentiable $2$-chain of $\C^*$ such that
\[
        [c]\in H_2(\C^*,\C^*\setminus K)
\]
is the relative orientation class $\alpha_{\C^*, K}$, which is the image of the orientation class $\alpha_{\C^*}$ by the canonical morphism
\[
        {}^{\BM}\!H_2(\C^*)\to  H_2(\C^*,\C^*\setminus K).
\]
Then, on $U$, one has
\[
        \overline{\partial}\underline{\omega}_1 \star \overline{\partial}\underline{\omega}_2
        = \left(\int_{c}
                \frac{\partial\underline{f}_1}{\partial\overline{z}}(\zeta)\frac{\partial\underline{f}_2}{\partial\overline{z}}\left(\frac{z}{\zeta}\right)
                \frac{d\overline{\zeta}}{\overline{\zeta}}\wedge\frac{d\zeta}{\zeta}
          \right) d\overline{z}\wedge dz,        
\]
since the integrated form is supported by $\underline{S}_1\cap z\underline{S}_2^{-1}\subset K$ for any $z\in U$. Moreover, the function $\underline{f}_2$ is infinitely differentiable on $\C^*$ and the chain $c$ is supported by a compact subset of $\C^*$. Thus, the function
\[
        f : z \mapsto \int_{c}
                \frac{\partial\underline{f}_1}{\partial\overline{z}}(\zeta)\underline{f}_2\left(\frac{z}{\zeta}\right)
                d\overline{\zeta}\wedge\frac{d\zeta}{\zeta}
\]
is infinitely differentiable on $\C^*$ and
\[
        \frac{\partial f}{\partial\overline{z}}(z)
        =\int_{c}
                \frac{\partial\underline{f}_1}{\partial\overline{z}}(\zeta)\frac{\partial\underline{f}_2}{\partial\overline{z}}\left(\frac{z}{\zeta}\right)
                \frac{d\overline{\zeta}}{\overline{\zeta}}\wedge\frac{d\zeta}{\zeta}
\]
Therefore, on $U$, one has
\[
        \overline{\partial}\underline{\omega}_1 \star \overline{\partial}\underline{\omega}_2
        =\overline{\partial}\omega
\]
where $\omega=f(z)dz$. Since $\supp(\overline{\partial}\underline{\omega}_1 \star \overline{\partial}\underline{\omega}_2)\subset\underline{S}_1\underline{S}_2$, the function $f$ is holomorphic on $U\setminus\underline{S}_1\underline{S}_2$ and it follows from what precedes that
\[
        ([\omega_1]\star[\omega_2])|_{U} = [\omega|_{U}]
\]
in
\[
        \Omega(U\backslash \underline{S}_1\underline{S}_2)/\Omega(U) \simeq H^1_{(\underline{S}_1\underline{S}_2) \cap U}(U,\Omega_{\C^*}).
\]

\bigskip

Let us now show how to compute $[\omega|_{U}]$ in $\Omega(U\setminus \overline{V})/\Omega(U)$ by means of $f_1$ and $f_2$ alone. Since $V$ is an open neighbourhood of $\underline{S}_1\underline{S}_2$,
\[
        \underline{S}_1\cap(\overline{U}\setminus V)\underline{S}_2^{-1} = \emptyset.
\]
Therefore,
\[
        \C^*=(\C^*\setminus\underline{S}_1)\cup \left(\C^*\setminus\left((\overline{U}\setminus V)\underline{S}_2^{-1}\right)\right)
\]
and, replacing if necessary $c$ by a barycentric subdivision, we may assume that $c=c_1+c_2$ where
\[
        \supp c_1 \subset \C^*\setminus\underline{S}_1
        \quad\text{and}\quad
        \supp c_2 \subset \C^*\setminus\left((\overline{U}\setminus V)\underline{S}_2^{-1}\right).
\]
Since $\supp\frac{\partial\underline{f}_1}{\partial\overline{z}}\subset\underline{S}_1$, it is then clear that
\[
        f(z)=\int_{c_2}
                \frac{\partial\underline{f}_1}{\partial\overline{z}}(\zeta)\underline{f}_2\left(\frac{z}{\zeta}\right)
                d\overline{\zeta}\wedge\frac{d\zeta}{\zeta}.
\]
Moreover, for any $z\in\overline{U}\setminus V$ one has
\[
        \C^*\setminus z\underline{S}_2^{-1}
        \supset \C^*\setminus\left((\overline{U}\setminus V)\underline{S}_2^{-1}\right)
        \supset \supp c_2
\]
and since the function $\zeta\mapsto \underline{f}_2(z/\zeta)$ is holomorphic on $\C^*\setminus z\underline{S}_2^{-1}$, it follows that
\begin{align*}
f(z)
&=\int_{c_2}
  \frac{\partial}{\partial\overline{\zeta}}\left(\underline{f}_1(\zeta)\underline{f}_2\left(\frac{z}{\zeta}\right)\frac{1}{\zeta}\right)
  d\overline{\zeta}\wedge d\zeta\\
&=\int_{\partial c_2}\underline{f}_1(\zeta)\underline{f}_2\left(\frac{z}{\zeta}\right)\frac{d\zeta}{\zeta}.
\end{align*}
By construction,
\[
        \supp(\partial c) \subset \C^*\setminus K = (\C^*\setminus\underline{S}_1) \cup (\C^*\setminus\overline{U}\underline{S}_2^{-1}).
\]
Up to replacing $c$ by a one of its barycentric subdivisions, we may thus assume that $\partial c = c'_1+c'_2$ where $\supp c'_1\subset \C^*\setminus\underline{S}_1$ and $\supp c'_2\subset \C^*\setminus\overline{U}\underline{S}_2^{-1}$. Since $$\partial c_1+\partial c_2=\partial c= c'_1+ c'_2,$$ there is a chain $c_3$ such that
\[
        \partial c_2 -c'_2 = c_3 = c'_1 - \partial c_1.
\]
Since $\supp c'_2\subset \C^*\setminus\overline{U}\underline{S}_2^{-1}$, the function
\[
        z\mapsto \int_{c'_2}\underline{f}_1(\zeta)\underline{f}_2\left(\frac{z}{\zeta}\right)\frac{d\zeta}{\zeta}
\]
is clearly holomorphic on $U$. Hence, the image of $[\omega|_{U}]$ in $\Omega(U\backslash \overline{V})/\Omega(U)$ is $[g(z)dz]$ where $g$ is the holomorphic function on $U\setminus\overline{V}$ defined by setting
\[
        g(z)=\int_{c_3}\underline{f}_1(\zeta)\underline{f}_2\left(\frac{z}{\zeta}\right)\frac{d\zeta}{\zeta}.
\]
Since
\[
        \supp(\partial c_2 -c'_2)
        \subset \left(\C^*\setminus\left((\overline{U}\setminus V)\underline{S}_2^{-1}\right)\right) \cup \left(\C^*\setminus\overline{U}\underline{S}_2^{-1}\right)
        = \C^*\setminus\left((\overline{U}\setminus V)\underline{S}_2^{-1}\right)
\]
and
\[
        \supp(c'_1 - \partial c_1)
        \subset (\C^*\setminus\underline{S}_1) \cup (\C^*\setminus\underline{S}_1)
        = \C^*\setminus\underline{S}_1,
\]
it is clear that
\[
        \supp c_3 \subset (\C^*\setminus\underline{S}_1) \cap \left(\C^*\setminus\left((\overline{U}\setminus V)\underline{S}_2^{-1}\right)\right).
\]
Therefore, we have in fact
\[
        g(z)=\int_{c_3}f_1(\zeta)f_2\left(\frac{z}{\zeta}\right)\frac{d\zeta}{\zeta}
\]
for any $z\in U\setminus\overline{V}$. Moreover, since $\partial c_3 = \partial c'_1 = - \partial c'_2$, it is clear that $$\supp\partial c_3\subset (\C^*\setminus\underline{S}_1) \cap (\C^*\setminus\overline{U}\underline{S}_2^{-1}).$$
\noindent
So,
\[
        c_3 \in Z_1(\left(\C^*\setminus\underline{S}_1\right)\cap\left(\C^*\setminus\left((\overline{U}\setminus V)\underline{S}_2^{-1}\right)\right),
                            \C^*\setminus\underline{S}_1\cap\C^*\setminus\overline{U}\underline{S}_2^{-1})
\]
and it follows by the construction of the chain $c_3$ that its class in
\[
        H_1(\left(\C^*\setminus\underline{S}_1\right)\cap\left(\C^*\setminus\left((\overline{U}\setminus V)\underline{S}_2^{-1}\right)\right),
                \C^*\setminus\underline{S}_1\cap\C^*\setminus\overline{U}\underline{S}_2^{-1})
\]
coïncides with the image of orientation class of $\C^*$ by the following sequence of canonical morphisms
\begin{align*}
{}^{\BM}\!H_2(\C^*)
&\to H_2(\C^*\setminus(\overline{U}\setminus V)\underline{S}_2^{-1},
         (\C^*\setminus(\overline{U}\setminus V)\underline{S}_2^{-1})\setminus (\underline{S}_1\cap \overline{U}\underline{S}_2^{-1}))\\
&\to H_1((\C^*\setminus(\overline{U}\setminus V)\underline{S}_2^{-1})\setminus (\underline{S}_1\cap \overline{U}\underline{S}_2^{-1}))\\
&= H_1(((\C^*\setminus\underline{S}_1) \cap (\C^*\setminus(\overline{U}\setminus V)\underline{S}_2^{-1}))\cup (\C^*\setminus\overline{U}\underline{S}_2^{-1}))\\
&\to H_1(((\C^*\setminus\underline{S}_1) \cap (\C^*\setminus(\overline{U}\setminus V)\underline{S}_2^{-1}))\cup (\C^*\setminus\overline{U}\underline{S}_2^{-1}),
         \C^*\setminus\overline{U}\underline{S}_2^{-1})\\
&\to H_1((\C^*\setminus\underline{S}_1) \cap (\C^*\setminus(\overline{U}\setminus V)\underline{S}_2^{-1}),
                          (\C^*\setminus\underline{S}_1)\cap (\C^*\setminus\overline{U}\underline{S}_2^{-1})).
\end{align*}

This clearly shows that $c_3$ is a relative Hadamard cycle for $\underline{S}_1$ with respect to $\overline{U}\underline{S}_2^{-1}$ in $\C^*\setminus (\overline{U}\setminus V)\underline{S}_2^{-1}$. Thus, $c_3$ is also a relative Hadamard cycle for $S_1$ with respect to $\overline{U}S_2^{-1}$ in $\C^*\setminus (\overline{U}\setminus V)S_2^{-1}$.

\bigskip
To conclude, it remains to show that if $c'_3$ is another relative Hadamard cycle for $S_1$ with respect to $\overline{U}S_2^{-1}$ in $\C^*\setminus (\overline{U}\setminus V)S_2^{-1}$ and if
\[
        \check{g}(z)=\int_{c'_3}f_1(\zeta)f_2\left(\frac{z}{\zeta}\right)\frac{d\zeta}{\zeta}
\]
for any $z \in U\setminus\overline{V}$, then $[g(z)dz]=[\check{g}(z)dz]$ in $\Omega(U\backslash \overline{V})/\Omega(U)$. For such a $c'_3$, we have $[c_3]=[c'_3]$ in
\[
H_1((\C^*\setminus\underline{S}_1) \cap (\C^*\setminus(\overline{U}\setminus V)\underline{S}_2^{-1}),
                          (\C^*\setminus\underline{S}_1)\cap (\C^*\setminus\overline{U}\underline{S}_2^{-1})).        
\]
Therefore, $c'_3=c_3+c_4+\partial c_5$ where $c_4$ is a $1$-chain of $(\C^*\setminus\underline{S}_1)\cap (\C^*\setminus\overline{U}\underline{S}_2^{-1})$ and $c_5$ is a $2$-chain of $(\C^*\setminus\underline{S}_1) \cap (\C^*\setminus(\overline{U}\setminus V)\underline{S}_2^{-1})$. It follows that the function
\[
        \check{g} : z\mapsto \int_{c'_3}f_1(\zeta)f_2\left(\frac{z}{\zeta}\right)\frac{d\zeta}{\zeta}
\]
is a holomorphic function on $U\setminus\overline{V}$ and that
\[
        \check{g}(z)=g(z)+\int_{c_4}f_1(\zeta)f_2\left(\frac{z}{\zeta}\right)\frac{d\zeta}{\zeta}
\]
on $U\setminus\overline{V}$. Since
\[
        z\mapsto \int_{c_4}f_1(\zeta)f_2\left(\frac{z}{\zeta}\right)\frac{d\zeta}{\zeta}
\]
is clearly holomorphic on $U$, we have $[g(z)dz]=[\check{g}(z)dz]$ in $\Omega(U\backslash \overline{V})/\Omega(U)$ as expected.
\end{proof}

\section{The case of strongly convolvable sets}

It is natural to ask if one can compute the holomorphic cohomological multiplicative convolution on $\C^*$ thanks to a global formula, by adding extra-conditions on $S_1$ and $S_2$. Recalling Definition~\ref{def:closedstareligible}, we are led to introduce the following one :

\begin{definition}
Let $S_1$ and $S_2$ be two convolvable proper closed subsets of $\C^*$ such that $S_1S_2 \neq \C^*$. These two closed sets are said to be \emph{strongly convolvable} if, furthermore, $\overline{S}_1$ and $\overline{S}_2$ are star-eligible, that is to say, if $\overline{S}_1 \times \overline{S}_2 \subset M.$ (Here $\overline{(.)}$ denotes the closure in $\PPP.$)
\end{definition}

\begin{remark}
One can find convolvable subsets of $\C^*$ which are not strongly convolvable. For example, consider
\[
S_1 = \{(2m)! : m \in \N\} \qquad \text{and} \qquad S_2 = \left\{ \frac{1}{(2n+1)!} : n\in \N\right\}.
\]
\end{remark}

We shall now highlight the link with the generalized Hadamard product. Recall Definitions~\ref{def:genhadcycle} and \ref{def:genhadproduct}.

\begin{proposition}\label{prop:strongly}
Let $S_1$ and $S_2$ be two strongly convolvable proper closed subsets of $\C^*.$ Assume that $\omega_1 = f_1 dz$ and $\omega_2 = f_2 dz$ with $f_1 \in \Hol(\C^* \backslash S_1)$ and $f_2 \in \Hol(\C^* \backslash S_2)$. For all $z\in \C^* \backslash S_1 S_2$, let $c_z$ be a generalized Hadamard cycle for $\overline{S}_1$ in $\PPP\backslash (z\overline{S}_2^{-1} \cup (\{0,\infty\}\backslash \overline{S}_1)).$
Then

\[
[\omega_1] \star [\omega_2] = [f dz] \in \Omega(\C^* \backslash S_1 S_2)/\Omega(\C^*),
\]
where 
\[
f(z) = -\int_{c_z} f_1(\zeta) f_2\left(\frac{z}{\zeta}\right) \frac{d\zeta}{\zeta},
\]
for all $z\in \C^*\backslash S_1S_2.$
\end{proposition}

\begin{proof}
Let $U$ be a relatively compact open subset of $\C^*$ and $V$ an open neighbourhood  of $S_1 S_2$ in $\C^*$. Let $c$ be a relative Hadamard cycle for $S_1$ with respect to $\overline{U}S_2^{-1}$ in $\C^* \backslash (\overline{U} \backslash V)S_2^{-1}$. Then, by a similar argument as in the proof of Lemma~\ref{lem:uniform}, it is clear that the image of $[c_z]$ by the sequence of canonical maps

$$
\xymatrix{H_1(\PPP\backslash (\overline{S}_1 \cup z\overline{S}_2^{-1} \cup \{0\} \cup \{\infty\})) = H_1(\C^* \backslash (S_1 \cup zS_2^{-1})) \ar[d]\\
{}^{\BM}\!H_1(\C^* \backslash (S_1 \cup zS_2^{-1})) \ar[d]\\
{}^{\BM}\!H_1((\C^*\setminus S_1) \cap (\C^*\setminus(\overline{U}\setminus V)S_2^{-1}))\ar[d]\\
H_1((\C^*\setminus S_1) \cap (\C^*\setminus(\overline{U}\setminus V)S_2^{-1}),
                          (\C^*\setminus S_1)\cap (\C^*\setminus\overline{U}S_2^{-1}))
}
$$

\noindent
is $[-c]$ for all $z\in \overline{U}\backslash V.$ Hence 

\[
\int_{c} f_1(\zeta) f_2\left(\frac{z}{\zeta}\right) \frac{d\zeta}{\zeta} = -\int_{c_z} f_1(\zeta) f_2\left(\frac{z}{\zeta}\right) \frac{d\zeta}{\zeta}, \quad \forall z \in U \backslash \overline{V}.
\]
Since this argument is valid for all $U$ and all $V$, the conclusion follows from Theorem~\ref{thm:cstarconvol}.
\end{proof}
\noindent
In this context, we set $(f_1\star f_2)(z) = \int_{c_z} f_1(\zeta) f_2\left(\frac{z}{\zeta}\right) \frac{d\zeta}{\zeta}.$ If $f_1 \in \Hol(\PPP \setminus \overline{S}_1)$ and $f_2 \in \Hol(\PPP \setminus \overline{S}_2),$ this really coincides with the generalized Hadamard product.

\begin{remark}
Let $S_1$ and $S_2$ be two strongly convolvable proper closed subsets of $\C^*.$ Let us make an identification $f dz \leftrightarrow -2i\pi f$ between holomorphic $1$-forms and holomorphic functions. Then, by the previous proposition, the holomorphic cohomological convolution morphism 

\[
H^1_{S_1}(\C^*, \Omega_{\C^*}) \otimes H^1_{S_2}(\C^*, \Omega_{\C^*}) \to H^1_{S_1 S_2}(\C^*, \Omega_{\C^*})
\]
can be seen as a bilinear map

\[
\Hol(\C^*\backslash S_1)/\Hol(\C^*) \times \Hol(\C^*\backslash S_2)/\Hol(\C^*) \to \Hol(\C^*\backslash S_1 S_2)/\Hol(\C^*),
\]
which can be computed by 

\[
[f_1]\star [f_2] = [f_1 \star f_2].
\]
\end{remark}

For the following example, we use the notation $D(0,R) = \{z\in \C : |z|<R\}$ with $R>0.$

\begin{example}\label{ex:hadamard}
Let $S = \C^*\backslash D(0,s)$ and $T = \C^*\backslash D(0,t)$ with $s>0$, $t>0$ and let 
\[
f \in \Hol(\C^* \backslash S)= \Hol(D(0,s)\backslash \{0\})\quad\text{and}\quad g \in \Hol(\C^* \backslash T) = \Hol(D(0,t)\backslash \{0\})
\] 
be two holomorphic functions. Then, $S$ and $T$ are strongly convolvable proper closed subsets of $\C^*$ and  we can write $f(z) = \sum_{n=-\infty}^{+\infty} a_n z^n$, $g(z) = \sum_{n=-\infty}^{+\infty} b_n z^n$. Since the polar part of $f$ (resp. $g$) is holomorphic on $\C^*$, we have
$
[f]=\left[\sum_{n=0}^{+\infty} a_n z^n \right]$ in $\Hol(D(0,s)\backslash \{0\})/\Hol(\C^*)$
and
$
[g]=\left[\sum_{n=0}^{+\infty} b_n z^n \right]$ in $\Hol(D(0,t)\backslash \{0\})/\Hol(\C^*).
$
Using the preceding remark, we see that the holomorphic cohomological convolution $[f]\star[g]$ is given by
\[
[f\star g] = \left[\sum_{n=0}^{+\infty} a_n b_n z^n \right],
\]
since the generalized Hadamard product coincides with the usual one in this case.  
\end{example}

Let us now state a trivial proposition :

\begin{proposition}\label{prop:restriction}
Let $S_1$ and $S_2$ be two convolvable closed subsets of $\C^*$ and $S'_1\subset S_1$, $S'_2 \subset S_2$ two closed subsets. Then, $S'_1$ and $S'_2$ are convolvable and the diagram

\[
\xymatrix{ H^1_{S_1}(\C^*,\Omega_{\C^*}) \otimes H^1_{S_2}(\C^*,\Omega_{\C^*}) \ar[r] & H^1_{S_1 S_2}(\C^*,\Omega_{\C^*}) \\ 
H^1_{S'_1}(\C^*,\Omega_{\C^*}) \otimes H^1_{S'_2}(\C^*,\Omega_{\C^*}) \ar[u] \ar[r] & H^1_{S'_1 S'_2}(\C^*,\Omega_{\C^*}) \ar[u] 
} 
\]
where the horizontal arrows are given by the holomorphic cohomological convolution morphisms, is commutative.  
\end{proposition}

Example~\ref{ex:hadamard} combined with Proposition~\ref{prop:restriction} allows to compute several other examples.

\begin{example}
Let $S_1 = S_2 = (-\infty,-1].$ The principal determination of the function $z\mapsto\ln(1+z)$ is holomorphic on $\C^* \backslash S_1.$ Moreover, $S_1$ and $S_2$ are strongly convolvable and thus, there is $g\in \Hol(\C^* \backslash [1,+\infty))$ such that

\[
[\ln(1+z)] \star [\ln(1+z)] = [g].
\]
Using the previous results, one has 

\begin{align*}
([\ln(1+z)] \star [\ln(1+z)])|_{D(0,1)} & = [\ln(1+z)|_{D(0,1)}] \star [\ln(1+z)|_{D(0,1)}]\\
& = \left[\sum_{n=1}^{+\infty} \frac{(-1)^{n+1}}{n} z^n \right] \star \left[\sum_{n=1}^{+\infty} \frac{(-1)^{n+1}}{n} z^n \right]\\
& = \left[\sum_{n=1}^{\infty} \frac{z^n}{n^2}\right]\\
& = [\text{Li}_2(z)]|_{D(0,1)},
\end{align*}
where $\text{Li}_2$ is the principal dilogarithm function, holomorphic on $\Hol(\C^* \backslash [1,+\infty))$. Hence, there is $h\in \Hol(\C^*)$ such that 
\[
g|_{D(0,1)}-\text{Li}_2|_{D(0,1)} = h.
\]
By the uniqueness of the analytic continuation, one deduces that $g-\text{Li}_2 =h$ on $\C^* \backslash [1,+\infty)$ and, thus, that 
\[
[\ln(1+z)] \star [\ln(1+z)] = [\text{Li}_2(z)]
\] 
in $\Hol(\C^* \backslash S_1 S_2)/\Hol(\C^*).$
\end{example}

\end{document}